\documentclass[12pt]{article}
\usepackage{epsfig}
\usepackage{latexsym}
\usepackage{graphicx}
\usepackage{amsmath}
\usepackage{amsfonts}
\usepackage{amssymb}
\usepackage{color}

\newtheorem{theorem}{Theorem}

\newtheorem{lemma}{Lemma}
\newtheorem{corollary}{Corollary}
\newtheorem{conjecture}{Conjecture}

\newtheorem{observation}{Observation}

\newcommand{\per}{{\rm per}}
\newcommand{\pnz}{{\rm permanent-non-singular }}

\newenvironment{proof}{
\par
\noindent {\bf Proof.}\rm}{\mbox{}\hfill\rule{0.5em}{0.809em}\par}

\begin{document}
\title{Total weight choosability for Halin graphs}
\author{ Yu-Chang Liang \thanks{Department of Applied Mathematics, National Pingtung University,  Pingtung,
   Taiwan 90003. Grant number: MOST 105-2811-M-153-001. Email: chase2369219@hotmail.com}
   \and Tsai-Lien Wong \thanks{Department
of Applied Mathematics, National Sun Yat-sen University, Kaohsiung,
Taiwan 80424. Grant numbers: 105-2918-I-110-003. Email:
tlwong@math.nsysu.edu.tw}
\and Xuding Zhu
\thanks{Department of Mathematics, Zhejiang Normal University,
China. Grant numbers:
  NSF11171310 and ZJNSF  Z6110786.  Email: xudingzhu@gmail.com. }
        \\[0.2cm]
       }

\date{\today}

\maketitle
\begin{abstract}
A proper total weighting of a graph $G$ is a mapping $\phi$ which assigns to each
vertex and each edge of $G$ a real number as its weight so that
for any edge $uv$ of $G$, $\sum_{e \in E(v)}\phi(e)+\phi(v) \ne \sum_{e \in E(u)}\phi(e)+\phi(u)$.
A $(k,k')$-list assignment of $G$ is a mapping $L$ which assigns to each vertex $v$ a set $L(v)$ of
$k$ permissible weights and to each edge $e$ a set $L(e)$ of $k'$ permissible weights.
An $L$-total weighting is a total weighting $\phi$ with $\phi(z) \in L(z)$ for each $z \in V(G) \cup E(G)$.
A graph $G$ is called  $(k,k')$-choosable if for every $(k,k')$-list assignment $L$ of $G$, there exists a
proper $L$-total weighting.   As a strenghtening of the well-known 1-2-3 conjecture,  it was
conjectured in [ Wong and   Zhu,   Total weight choosability of graphs, J. Graph Theory 66 (2011), 198-212] that every graph without  isolated edge is
$(1,3)$-choosable.  It is easy to verified this conjecture for trees, however, to prove it for wheels seemed to be quite non-trivial.  In this paper, we develop some tools  and techniques  which enable us to prove this conjecture for generalized  Halin graphs.
\end{abstract}

{\small \noindent{{\bf Key words: }  Total weighting, $(k,k')$-matrix,  Halin graphs}
\section{Introduction}

A {\em total weighting} of $G$ is a mapping
$\phi: V(G) \cup E(G) \to R$. A total weighting $\phi$ is {\em
proper} if for any edge $uv$ of $G$,
$$\sum_{e \in E(u)}\phi(e) + \phi(u) \ne \sum_{e\in E(v)}\phi(e) + \phi(v),$$
where $E(v)$ is the set of edges incident to $v$. A total weighting
$\phi$ with $\phi(v)=0$ for all vertices $v$ is also called an {\em
edge weighting}.

Proper edge weighting (also called vertex colouring edge weighting) of graphs was introduced in \cite{KLT2004}.
It was conjectured in \cite{KLT2004} that every graph with no isolated edges has a proper edge weighting
$\phi$ with   $\phi(e)  \in
\{1,2, 3\}$ for $e \in E(G)$. This conjecture,
now called the   {\em 1-2-3 Conjecture}, has received a
lot of attention \cite{Add2005, Add2007, chang, KLT2004, KKP10, yu,
PW2010, WY2008}. It still remains open, and  the best partial result
on this conjecture was proved in \cite{KKP10}: every graph with no isolated edge has a proper edge weighting $\phi$
with  $\phi(e) \in \{1,2,3,4,5\}$ for all $e \in E(G)$.

Proper total weighting was first studied   in \cite{PW2010}.
It was conjectured in \cite{PW2010} that   every graph has a
proper total weighting $\phi$ with $\phi(z) \in \{1,2\}$ for all $z
\in V(G) \cup E(G)$. This conjecture, now called the {\em 1-2 Conjecture} has also received a lot of attention
and the best partial result was proved in \cite{K2008}: for any graph $G$, there is a proper total weighting $\phi$
with $\phi(v) \in \{1,2\}$ for each vertex $v$ and
  $\phi(e) \in \{1,2,3\}$ for each $e \in E(G)$.

A total list assignment of $G$ is a mapping $L$ which assigns to each element $z \in V(G) \cup E(G)$
a set $L(z)$ of real numbers as permissible weights.
An $L$-total weighting is a total weighting $\phi$ with $\phi(z) \in L(z)$ for each $z \in V(G) \cup E(G)$.
Assume $\psi: V(G) \cup E(G) \to \{1,2,\ldots\}$ is a  mapping which assigns to each vertex or edge $z$ of $G$
a positive integer. A total list assignment $L$ of $G$ is called a {\em
$\psi$-total list assignment} of $G$ if $|L(z)| = \psi(z)$ for all $z \in V(G) \cup E(G)$.
A graph $G$ is called {\em $\psi$-choosable} if for every $\psi$-list assignment $L$ of $G$, there exists a
proper $L$-total weighting. A graph $G$ is called  $(k,k')$-choosable if $G$ is $\psi$-choosable, where $\psi(v) = k$ for each vertex $v$ and $\psi(e)=k'$ for each edge $e$.


The list version of total weighting are studied in a few papers
\cite{BGN09, PY2012, PW2011, WWZ2012, WYZ08, WZ11} 
 It is known \cite{WZ11} that $G$ is $(k,1)$-choosable if and only if
$G$ is (vertex) $k$-choosable. So the concept of
$(k,k')$-choosability is a common generalization of vertex
choosability, edge weighting and total weighting of graphs. As
strengthening of the 1-2-3 conjecture and the 1-2 conjecture, it was
conjectured in \cite{BGN09,WZ11} that  every graph with no isolated
edges is $(1,3)$-choosable and conjectured
in \cite{WZ11} that every graph is $(2,2)$-choosable.
These two conjectures are
called the {\em $(1,3)$-choosability conjecture}
 and the {\em $(2,2)$-choosability conjecture}, respectively.


In the study of total weighting of graphs, one main algebraic tool  is Combinatorial Nullstellensatz.

 For each $z  \in V(G) \cup E(G)$, let $x_z$
be a variable associated to $z$. Fix an arbitrary orientation $D$ of $G$.
Consider the
polynomial
$$P_G(\{x_z: z \in V(G) \cup E(G)\}) = \prod_{e=uv \in E(D)}\left(  \left(\sum_{e \in E(u)} x_e+ x_u\right) - \left(\sum_{e \in E(v)} x_e+ x_v\right)\right).$$
Assign a real number $\phi(z)$ to the variable $x_z$, and view
$\phi(z)$ as the weight of $z$.
Let $P_G( \phi  )$ be the evaluation of the
polynomial at $x_z = \phi(z)$. Then $\phi$ is a proper total
weighting of $G$ if and only if $P_G( \phi) \ne 0$.
The question  is under what condition one can find an assignment $\phi$
for which $P_G( \phi) \ne 0$.

An {\em index function} of $G$ is a mapping $\eta$ which
assigns to each vertex or edge $z$ of $G$ a non-negative integer $\eta(z)$.
An index function $\eta$ of $G$ is {\em valid} if $\sum_{z \in V \cup E}\eta(z) = |E|$.
Note that $|E|$ is the degree of the polynomial
$P_G(\{x_z: z \in V(G) \cup E(G)\})$.
For a valid index function $\eta$, let
$c_{\eta}$ be the coefficient of the monomial
 $\prod_{z \in V \cup E} x_{z}^{\eta(z)}$ in the expansion of $P_G$.
It follows from the Combinatorial Nullstellensatz
\cite{nullstellensatz,AlonTarsi} that if $c_{\eta} \ne 0$, and
$L$ is a list assignment which assigns to each $z \in V(G) \cup E(G)$ a
set $L(z)$ of  $ \eta(z)+1$ real numbers, then  there exists a
mapping $\phi$ with  $\phi(z) \in L(z)$   such that $$P_G(\phi) \ne
0.$$

 Therefore, to prove that a graph $G$ is $(k,k')$-choosable, it suffices to show that
there exists an index function $\eta$ with $\eta(v) \le k-1$ for each vertex $v$ and $\eta(e) \le k'-1$ for
each edge $e$ and $c_{\eta} \ne 0$.

 The coefficient $c_{\eta}$ is related to the permanent of the martix   below (see Equation (1)).

We write the polynomial $P_G(\{x_z: z \in V(G) \cup E(G)\})$  as
$$P_G(\{x_z: z \in V(G) \cup E(G)\}) = \prod_{e \in E(D)}\sum_{z \in V(G) \cup E(G)}A_G[e,z]x_z.$$
Then for $e \in E(G)$ and $z \in V(G) \cup E(G)$,
if $e=(u,v)$ (oriented from $u$ to $v$), then
\begin{equation*}
A_G[e,z]=
\begin{cases} 1 & \text{if $z=v$, or $z \ne e$ is an edge incident to $v$,}
\\
-1 & \text{if $z=u$, or $z \ne e$ is an edge incident to $u$,}
\\
0 &\text{otherwise.}
\end{cases}
\end{equation*}
Now $A_G $ is a matrix, whose rows  are indexed by the edges of $G$ and the columns are indexed by edges and vertices of $G$.
Given a vertex or edge $z$ of $G$, let $A_G(z)$ be the column of $A_G$ indexed by $z$.
For an index function $\eta$ of $G$, let $A_G(\eta)$ be the
 matrix, each of
its column is a column of $A_G$, and each column $A_G(z)$ of $A_G$
occurs $\eta(z)$ times as  a column of $A_G(\eta)$. It is known
\cite{alontarsi1989}
and easy to verify that for a valid
index function $\eta$ of $G$,

$$c_{\eta} =\frac{1}{\prod_{z \in V \cup E}\eta(z)!}\per(A_G(\eta)), \eqno(1)$$
 where $\per(A)$ denotes the permanent of
the square matrix $A$.  Recall that if
$A$ is an $m \times m$ matrix, then   $$\per(A) = \sum_{\sigma \in
S_m}A[i, \sigma(i)],$$ where $S_m$ is the symmetric group of order
$m$.

 A square matrix $A$ is {\em \pnz} if $\per (A) \ne 0$.
 A square matrix of the form $A_G(\eta)$ is called an $(a,b)$-matrix
 if $\eta(v) \le a$ for each vertex $v$ and $\eta(e) \le b$ for each edge $e$.
Motivated by an edge weighting and an total weighting problem of
graphs, the following two conjectures were proposed in
 \cite{BGN09} and \cite{WZ11}, respectively.
\begin{conjecture}
\label{guess1} Every graph $G$ has a \pnz $(1,1)$-matrix.
\end{conjecture}

\begin{conjecture}
\label{guess2} Every graph $G$ without isolated edges has a \pnz
$(0,2)$-matrix.
\end{conjecture}

  Conjecture \ref{guess1} and Conjecture \ref{guess2} have been studied in many papers (see \cite{Survey}  for a survey of partial results on these two conjectures), and both conjectures remain largely open.  It is easy to verify both conjectures for trees. However, proving these two conjectures for wheels seem to be quite non-trivial.
It was proved in \cite{ PW2011} that Conjecture \ref{guess1}   is true for wheel, and   in  \cite{ WZ2017}  for Halin graphs.  Quite surprising, Conjecture \ref{guess2} remained open for wheels for a long time.  In this paper, we develop some tools and techniques and settle Conjecture \ref{guess2} for generalized  Halin graphs.

\section{ Main theorem and some observations}

A {\em Halin graph} is a planar graph obtained by taking a plane
tree $T$ (an embedding of a tree on the plane) without degree $2$
vertices by adding a cycle connecting the leaves of the tree
cyclically. If the tree $T$ is allowed to have degree $2$ vertices, then the resulting
graph is called a {\em generalized Halin graph}.


\begin{theorem}\label{main}
\label{generalhalin} Every generalized Halin graph $G$ has a \pnz
$(0,2)$-matrix.
\end{theorem}

We will prove this theorem in the next two sections. In the proof, we shall frequently use the following  observations:

\begin{observation}
	\label{obs-linear}
	If $A$ is a matrix whose columns are integral liner combinations of columns of $A_G$ and $\per(A) \ne 0$ and each olumn $A_G(z)$ occurs in at most $\eta(z)$ times in the combinations, then there is an index function $\eta'$ with $\eta'(z) \le \eta(z)$ and $\per(A_G(\eta')) \ne 0$. Moreover, if $\per(A) \ne 0 \pmod{p}$ for some prime $p$, then $\per(A_G(\eta')) \ne 0 \pmod{p}$.
\end{observation}
This can be derived directly from  the multilinear property of permanent.

\begin{observation}\label{ob2}(  \cite{WZ11} )
For an edge $e=uv$   of $G$,
$$A_G(e)=A_G(u)+A_G(v). \eqno(2)$$
\end{observation}
The above  follows easily from the definition  of the matrix $A_G$ (cf.  \cite{WZ11}):

A {\em balloon} is a graph obtained by attaching a path to a cycle
(i.e., identify one end vertex of a path with a vertex of a cycle).
If the cycle is of odd length, then the ballon is called an odd balloon.
The path could be a single vertex, in which case the balloon is simply a cycle.
If the path is not a single vertex, then the unique vertex of degree $1$ is called the {\em root of the ballon}. Otherwise, the root of the ballon (which is a cycle) is an arbtriary vertex of the cycle.

\begin{observation}
	\label{obs-oddballon}
	If $B$ is an odd balloon with root $v$, then $2A_G(v)$ is an integral linear combination of $A_G(e)$ for $e \in E(B)$.
\end{observation}

Indeed, if $P=(v_1,v_2, \ldots, v_k)$ and $C=(u_1, u_2, \ldots, u_{2p+1})$ and the balloon is obtained by identifying $v_k$ with $u_1$,  let  $e_i =v_iv_{i+1} $(for  $1\leq i \leq k-1  $) and  $e'_i =u_iu_{i+1}$ (for $1\leq i \leq 2p+1 $  and $u_{2p+2}=u_1$ ),
then
$$2A_G(v_1) = 2A_G(e_1) -   \ldots + (-1)^{k} 2A_G(e_{k-1}) + (-1)^{k -1} (A_G(e'_1) -\ldots + A_G(e'_{2p+1})).$$

\section{Non-bipartite generalized Halin graphs}

In this section,  we consider non-bipartite generalized Halin graphs.

\begin{lemma}
\label{vertexcolumn} Let $G$ be a connected non-bipartite graph. If
there is a matrix $A$ whose columns consists of vertex columns of
$A_G$ and $\per(A) \ne 0 \pmod{p}$ for some odd prime $p$, then $G$ has
a \pnz $(0, p-1)$-matrix.
\end{lemma}
\begin{proof}
	Since $p$ is an odd prime, by replacing each column $A_G(v)$ with $2A_G(v)$ in $A$, the resulting matrix $A'$ has $\per(A') = 2^n \per(A) \ne 0 \pmod{p}$.
Since $G$ is connected and non-bipartite, for any vertex $v$, there is an odd balloon $B$ of $G$ with root $v$. By Observation \ref{obs-oddballon},  $2 A_G(v)$ can be written as an integral linear
combination of edge columns of $G$.
By Observation \ref{obs-linear},
there is an index function $\eta$ with $\eta(v) = 0$ for each vertex $v$ such that $\per(A_G(\eta)) \ne 0 \pmod{p}$. It is obvious that if
$\eta'$ is an index function for which $\eta'(e) \ge p$ for some edge $e$, then $\per(A_G(\eta')) = 0 \pmod{p}$. Therefore $\eta(e) \le p-1$ for each $e \in E(G)$. I.e., $A_G(\eta)$ is a \pnz $(0,p-1)$-matrix of
$G$.
\end{proof}

\begin{corollary}
\label{2deg} If $p$ is a prime, $G$ is a connected non-bipartite
$(p-1)$-degenerate graph, then $G$ has a \pnz $(0,p-1)$-matrix.
\end{corollary}
\begin{proof}
Order the vertices $v_1, v_2, \ldots, v_n$ in such a way that each
vertex $v_i$ has back-degree $d_i \le p-1$, i.e., $v_i$ has at
most $p-1$ neighbours $v_j$ with $j < i$. Let $A$ be the matrix
consisting $d_i$ copies of the column of $A_G$ indexed by $v_i$ for
$i=1,2,\ldots, n$. It is easy to verify that $|\per(A)| =
\prod_{i=1}^nd_i!$.
Hence $\per(A) \ne 0 \pmod{p}$. It follows from
Lemma \ref{vertexcolumn} that $G$ has a \pnz $(0,p-1)$-matrix.
\end{proof}

\begin{lemma}
\label{2deg2} Assume $p$ is an odd  prime, $G$ is a connected non-bipartite graph, $v$ is a vertex of $G$ of degree $d$ and $G-v$ has a \pnz $(0,p-1)$-matrix $A_{G-v}(\eta)$. If there are $t$ edge disjoint odd balloons $B_1, B_2, \ldots, B_t$ with root $v$ such that for any $1 \le i \le t$ and $e \in B_i$, $\eta(e)=0$ and $t \ge d/(p-1)$, then   $G$ has a \pnz $(0,p-1)$-matrix.
\end{lemma}
\begin{proof}
 Let $\eta'(z) =\eta(z)$ except that $\eta'(v) = d$.
	Let $A'$ be obtained from $A_G(\eta')$ by replacing each copy of $A_G(v)$ by $2 A_G(v)$. 	Then   $per( A_G(\eta'))=  2^d d! per(A_{G-v}(\eta)) \neq 0$.
By Observation \ref{obs-oddballon},
each copy of  $2A_G(v)$ can be written as integral linear combination of edge columns $A_G(e)$ for $e \in E(B_i)$, i.e., $\sum_{e \in E(B_i)}   a_{i,e} A_G(e)$ for each $i$, where $a_{i,e}$ are integers. As $t \ge d/(p-1)$,
we can replace the $d$ copies of    $2A_G(v)$ with
integral linear combinations $\sum_{e \in E(B_i)} a_{i,e}  A_G(e)$, so that each $B_i$ is used at most $p-1$ times. Therefore we can write each colummn of $A_G(\eta')$ as linear combination of edge columns of $A_G$, and each edge column is used at most $p-1$ times. So $G$ has a \pnz $(0,p-1)$-matrix.
\end{proof}

By Lemma \ref{vertexcolumn}, to prove that a non-bipartite generalized
Halin graph $G$ has a \pnz $(0,2)$-matrix, it suffices to show that
there is a matrix $A$ consisting of vertex columns of $A_G$ and
$\per(A) \ne 0 \pmod{3}$. By Equation (1), this is equivalent to
the existence of a valid index function $\eta$ of $G$ such that
$\eta(v)\leq 2$ for each vertex $v$, $\eta(e)=0$ for each edge $e$
and $c_{\eta} \ne 0 \pmod{3}$.

Recall that the graph polynomial of $G$ is defined as
$Q_G(\{x_v: v \in V(G)\}) = \prod_{uv \in E(\vec{G})}(x_u-x_v)$,
where $\vec{G}$ is an orientation of $G$. So $Q_G$ is obtained from $P_G$ by letting $x_e =0$ for each edge $e$ of $G$. Therefore,
if  $\eta(e)=0$ for all edges $e$, $c_{\eta}$ is indeed the coefficient of  the monomial $\prod_{v \in V(G)}x_v^{\eta(v)}$ in the expansion of the graph polynomial $Q_G$ of $G$.
For the purpose of calculating $c_{\eta}$ for such an index function $\eta$,
we use a result of  Alon and Tarsi
\cite{AlonTarsi}.

A sub-digraph $H$ (not necessarily connected) of a directed graph $D$
is called Eulerian if the in-degree $d^{-}_ H(v)$ of every vertex $v$
of $H$ is equal to its out-degree $d^{+}_ H(v)$. An Eulerian sub-digraph $H$ is even if it
has an even number of edges, otherwise, it is odd. Let $EE(D)$ and
$EO(D)$ denote the sets of even and odd Eulerian subgraphs of $D$,
respectively.   The following result was proved in
\cite{AlonTarsi}.

\begin{lemma}\cite{AlonTarsi}\label{eulerian}
Let $D = (V, E)$ be an orientation of an undirected graph $G$, and
$d_i$ is the out-degree of $v_i$ in $D$. Then the coefficient of
$\prod_{i=1}^n x_{v_i}^{d_i}$ in the graph polynomial of $G$ is
$\pm (|EE(D)|- |EO(D)|)$.
\end{lemma}

\begin{lemma} \label{nonbi}
	Let $G$ be a non-bipartite generalized Halin graph. Then $G$ has a \pnz
	$(0, 2)$-matrix.
\end{lemma}
\begin{proof}
Assume $G$ is obtained from a tree plane $T$ by adding edges connecting its leaves into a cycle $C$. We choose non-leave vertex of $T$ as the root of $T$.
If $T$ has an even number of leaves, then we orient the edges of $G$ in such a way that
the edges in the tree $T$ are all oriented towards to the root vertex, and orient the edges of $C$  so that it becomes a directed cycle.
In such an orientation $D$ of $G$, by repeated deleting sink vertices (that must isolated vertices in any Eulerian subgraph), the resulting graph is a directed even cycle $C$.    As
$D$ has  no odd Eulerain sub-digraph, and has $2$ even Eulerian sub-digraph (the empty digraph and $C$).
As each vertex has out-degree at most $2$.
The conlcusion follows from Lemmas  \ref{vertexcolumn},  \ref{eulerian} and Observation \ref{ob2}.

Assume $T$ has an odd number of leaves. Hence $C$ is an odd cycle.

Assume first that $G$ is not a wheel.  Let $v$ be a non-leaf vertex of $T$ all its sons are leaves.
Assume $v$ has $k$ leaf sons $v_1, v_2, \ldots, v_k$.

If $k$ is even, then
we orient  the cycle $C$ as a directed cycle, orient the tree $T$ with all edges
towards the root, except that the edge $vv_k$ is oriented from $v$ to $v_k$. Straightforward counting
shows that among Eulerian sub-digraphs of $D$ containing the directed edge $vv_k$,
  $k/2$ are odd and $k/2-1$ are even.
The empty   Eulerian subgraph is even, and the directed cycle is odd. Hence $|EE(D)|- |OE(D)|\ \ne 0 \pmod{3}$. As
each vertex has out-degree at most $2$,  we are done.

If $k$ is odd, then we oriented the edges of $T$ as in the case that $k$ is even,
except that the edge in $C$ oriented towards $v_k$ is reversed as an edge oriented away from $v_k$ (so $v_k$ becomes a source vertex in the cycle $C$).
Straightforward counting shows that among Eulerian sub-digraphs of $D$ containing the directed edge $vv_k$, $(k-1)/2$ are even and $(k-1)/2$
are odd.

There is one even Eulerian sub-digraph not using the edge $vv_k$ (the empty sub-digraph)
and no odd Eulerian sub-digraph not using the edge $vv_k$. Again
each vertex has out-degree at most $2$,  we are done.

Assume $G$ is an odd wheel with  $V(G)=\{w,v_1,v_2,\ldots v_n\}$,
and $w$ is the center of the wheel. If $n \le 5$, then it can be checked directly that $G$ has a \pnz $(0,2)$-matrix.
Assume $n \ge 7$. Consider the graph $G-v_n$.  We order the vertices of $G-v_n$ as $v_1,w,v_2\ldots, v_{n-1}$. Then each vertex has back-degree at most $2$. As in the proof of Corollary \ref{2deg}, for the index function $\eta$ with $\eta(w)=1, \eta(v_i)=2$ for $i=2,3,\ldots, n-1$,
$\per(A_G(\eta)) \ne 0 \pmod{3}$. It is easy to check that each vertex $w, v_2, v_3, \ldots, v_{n-1}$ is the root of an odd balloon in $G-v_n$  that does not contain any edge incident to $v_1$, and does not contain
the edges $v_{n-1}w$ and   $v_{2}w$ .
  By Observation \ref{obs-oddballon} (cf. the proof of Lemma \ref{vertexcolumn}), we know that there is an index function $\eta'$ of $G-v_n$ with $\per(A_{G-v_n}(\eta')) \ne 0 \pmod{3}$ such that $\eta'(v)=0$ for all $v \in V(G-v_n)$, $\eta'(e) \le 2$ for any edge $e$ of $G-v_n$, and $\eta'(e)=0$ for $e \in E(v_n) \cup E(v_1) \cup \{v_{n-1}w,  v_{2}w\}$.
  Now the vertex $v_n$ is the root of two edge disjoint odd balloons $B_1$ with $V(B_1) = \{v_n, w, v_{n-1}\}$ and $B_2$ with $V(B_2) = \{v_n, v_1, v_2, w\}$. As $2 \ge d_G(v_n)/2$, by Lemma \ref{2deg2}, $G$ has a  \pnz $(0,2)$-matrix.
\end{proof}

\section{Bipartite generalized Halin graphs}

\begin{lemma}
	\label{lem-2deg2} If $p$ is a prime, $G$ is a connected  bipartite
	$(p-1)$-degenerate graph, $v$ is a vertex of degree $1$, then $G$ has a \pnz  matrix in which each edge column occurs at most $p-1$ times, the vertex column indexed by $v$ occurs once and there are no other vertex column.
\end{lemma}
\begin{proof}
	Order the vertices $v_1, v_2, \ldots, v_n$ in such a way that each
	vertex $v_i$ has back-degree $d_i \le p-1$, i.e., $v_i$ has at
	most $p-1$ neighbours $v_j$ with $j < i$, and $v_n=v$.
	Let $A$ be the matrix
	consisting $d_i$ copies of the column of $A_G$ indexed by $v_i$ for
	$i=1,2,\ldots, n$. Similarly, $|\per(A)| =
	\prod_{i=1}^nd_i! \neq  0 \pmod{p}$.
	
	Assume   $A_G(v_i)$ is a column in $A$ indexed by $v_i$ and $i \ne n$. Let $A'$ be the matrix
	obtained from $A$ by replacing $A_G(v_i)$ by $ A_G(v)$.
	Since $A'$ has two copies of the column $A_G(v)$, which has
	only one nonzero entry, we know that $\per(A') = 0$. Therefore, if
	we replace $A_G(v_i)$ by $A_G(v_i) \pm A_G(v)$, the
	resulting matrix has the same permanent as $A$.
	
	For each vertex column in $A$ of the form $A_G(v_i)$ for $i \ne n$, we	replace it by $A_G(v_i) \pm A_G(v)$,
	where the $\pm$ sign is determined by the parity of
	the distance between the two vertices $v_i$ and $v$: if
	the distance is odd, then choose $+$, and otherwise choose $-$.
	Denote the resulting matrix by $A^*$. Then $\per(A^*)
	= \per(A') \ne 0 \pmod{p}$.
	
	Similarly as in the proof of Corollary \ref{2deg}, each column of $A^*$ other than the column indexed by
	$v$ can be written as an integral linear combination of edge columns
	of $A_{G}$.  As in the proof
	of Lemma \ref{vertexcolumn}, there is a matrix $A^{\#}$ consisting
	of edges columns of $A'_{G'}$ plus one column indexed by $v$, such
	that $\per(A^{\#}) \ne 0 \pmod{p}$, where each edge column   occurs at most $p-1$
	in $A^{\#}$.
\end{proof}

Assume $G$ is a graph and $X, Y$ are subsets of $V(G)$. We denote by $E[X,Y]$ the set of edges with one end vertex in $X$ and one end vertex with $Y$. Let $E[X] = E[X,X]$.

\begin{lemma}
\label{2degleft} Assume $G$ is a connected graph, and $V(G)=X \cup
Y$ is a partition of $G$.    If the subgraph $H$ induced by edges in $E[X] \cup E[X,Y]$ has a
 \pnz $(0,2)$ matrix $A$ which contains no columns indexed by edges in $E[X,Y]$ and  $G[Y]$  is $2$-degenerate, then $G$ has a \pnz $(0,2)$-matrix.
\end{lemma}
\begin{proof}
Assume $G[Y]$ has connected components $G[Y_1], G[Y_2], \ldots, G[Y_q]$. For each $1 \le i \le q$, let $e_i=x_iy_i$ be an edge connecting $x_i \in X$ to $y_i \in Y_i$. Let $G_i=G[Y_i]+e_i$.
By Lemma \ref{lem-2deg2}, $G_i$
has a \pnz  matrix $A_i$ in which each edge column occurs at most twice, the vertex column index by $x_i$ occurs once and there are no other vertex column.

Let $\eta_i$ be the index function of $G_i$
so that $A_i=A_{G_i}(\eta_i)$.

 Let $A'_i$ be the matrix obtained from $A_i$ by deleting the column indexed by $x_i$ and the row index by $e_i$. Since the column indexed by $x_i$ has only one nonzero entry, we conclude that $\per(A'_i) \ne 0$.

Let $A$ be a \pnz $(0,2)$ matrix of $H$ which contains no columns indexed by edges in $E[X,Y]$. Let $\eta$ be the index function of $H$ so that $A=A_H(\eta)$.

For each edge $e$ of $G$,
let
\[
\eta'(e)=\begin{cases}
\eta(e), &\text{if $e \in E[X]$}, \cr
 \eta_i(e), &\text{if $e \in E(G_i)$,} \cr
 0, & \text{ if $e \in E[X,Y] - \{e_1, e_2, \ldots, e_q\}$.}
 	\end{cases}
\]
 Let $A' = A_G(\eta')$. Note that $\eta'(e) \le 2$ for each edge $e$ of $G$. Now $A'$  is of the form

  \begin{eqnarray*}  A'= \left[
  	\begin{array}{cccccccc}
        &  &        &      &   \\
  		& A&        &      &   \\
  		&  &    A'_1&      &  {\raisebox{2ex}[0pt]{  *}} \\
  		&  &        &    \ddots \\
  		&  &       &        &   A'_q &  \\
        &  {\raisebox{2ex}[0pt]{\Huge0}}&       &        &
  	\end{array}\right]
  \end{eqnarray*}

Therefore $\per(A')=\per(A) \per(A'_1) \ldots \per(A'_q)\ne 0$, and hence $A'$ is
a \pnz $(0,2)$-matrix of $G$.
\end{proof}

Observe that any proper subgraph of a generalized Halin graph $G$ is $2$-degenerate. Therefore
to prove that a generalized Halin graph $G$ has a \pnz
$(0,2)$-matrix, by Lemma \ref{2degleft}, it suffices to find a
partition $V(G)=X \cup Y$ so that $ G[X] \cup E[X,Y]$ has a
 \pnz $(0,2)$ matrix which contains no columns indexed by edges in
$E[X,Y]$.

	Let $G$ be an oriented graph and $e$ be an edge in $G$. We call $e$
	a {\em sink edge} if all edges $e'$ adjacent to  $e$ are oriented
	towards $e$ (i.e, towards the common end vertex of $e$ and $e'$) and
	a {\em source edge} if all edges $e'$ adjacent to $e$ are oriented
	away from $e$.
	
\begin{lemma}
	\label{lem-new}
	Assume $G$ is a connected graph and $X \cup Y$ is a partition of  $V(G)$.
	If there is an orientation of edges in $E[X] \cup E[X,Y]$
	and  a mapping  $\phi: E[X] \cup E[X,Y] \to E[X]$
   such that for each $e \in E[X] \cup E[X,Y]$, $\phi(e) \ne e$ is   a source or a sink edge incident to $e$, and for each $e \in E[X]$,
 $|\phi^{-1}(e)| \le 2$,  then  the subgraph $H=G[X] \cup E[X,Y]$ has a
	 \pnz $(0,2)$ matrix which contains no columns indexed by edges in $E[X,Y]$.
\end{lemma}
\begin{proof}
Let $H$ be the subgraph of $G$ induced by edges in $E[X] \cup E[X,Y]$. Let $D$ be an orientation of edges in $H$, and  $\phi$ be a mapping
from $E[X] \cup E[X,Y] \to E[X]$
such that for each $e \in E[X] \cup E[X,Y]$, $\phi(e)$ is   a source or a sink edge incident to $e$, and for each $e \in E[X]$,
$|\phi^{-1}(e)| \le 2$.

 Let $\eta(e)=|\phi^{-1}(e)|$ for each edge $e \in E[X]$. We shall show that $A_H(\eta)$ has non-zero permanant.
 Note that the column vector $A_H(e)$ is non-negative if $e$   is a
 sink edge and   non-positive if $e$   is a source edge. Thus to prove that $A_H(\eta)$ has non-zero permanant, it suffices to find a one-to-one mapping $\pi$ between the rows and columns of $A_H(\eta)$
 such that for each row $e$, the entry $A_H(\eta)[e, \pi(e)] \ne 0$.
 The rows of $A_H(\eta)$ are indexed by edges in $G[X] \cup E[X,Y]$ and columns are indexed by a multiset of edges in $E[X]$, with each $e \in E[X]$ occurs $\eta(e) = |\phi^{-1}(e)|$ times. Since edges in $\phi^{-1}(e)$ are incident to $e$, the mapping $\phi$ is such a one-to-one mapping.	
\end{proof}

\begin{lemma}\label{bi}
	Let $G$ be a bipartite generalized Halin graph. Then $G$ has a \pnz
	$(0, 2)$-matrix.
\end{lemma}
\begin{proof}
By  Lemma \ref{2degleft}, it suffices to choose a set $X$ such that
the subgraph  $H=G[X] \cup E[X,Y]$  has a a \pnz
$(0,2)$ matrix which contains no columns indexed by edges in $E[X,Y]$.  In all the figures  below,  vertices of $X$ are indicated by open dots,    and vertices of $Y$ are indicated by solid dots.

Recall that $G$ is obtained from a plane tree $T$ by adding a cycle connecting its leaves in order. If $T$ is a path, by choosing  $X$ with three consecutive vertices and using twice of edges  in $E[X]$ as column vectors, then it can easily be verified that this is a  \pnz
	$(0, 2)$-matrix of
$H$. Assume $T$ is not a path. We choose a vertex $r \in V(T)$ of degree at least $3$ as the root of $T$. Let $v_1$ be a leaf with maximum depth. Since $G$ is bipartite, the father $v_2$ of $v_1$ has only one son (i.e. $d(v_2)=2$). Let $v_3$ be the father of $v_2$.

\bigskip
\noindent
{\bf Case 1:}  $v_3$ has two or three sons.
\bigskip

Let $w$ be  a leaf son of $v_3$ and choose $X=\{v_3,v_2,v_1,w\}$.  We orient the edges in $H$  so that $v_3$ is a sink vertex and $v_1$ is a source vertex.

\begin{figure}[h]
	\centering
	\scalebox{0.5}{\includegraphics{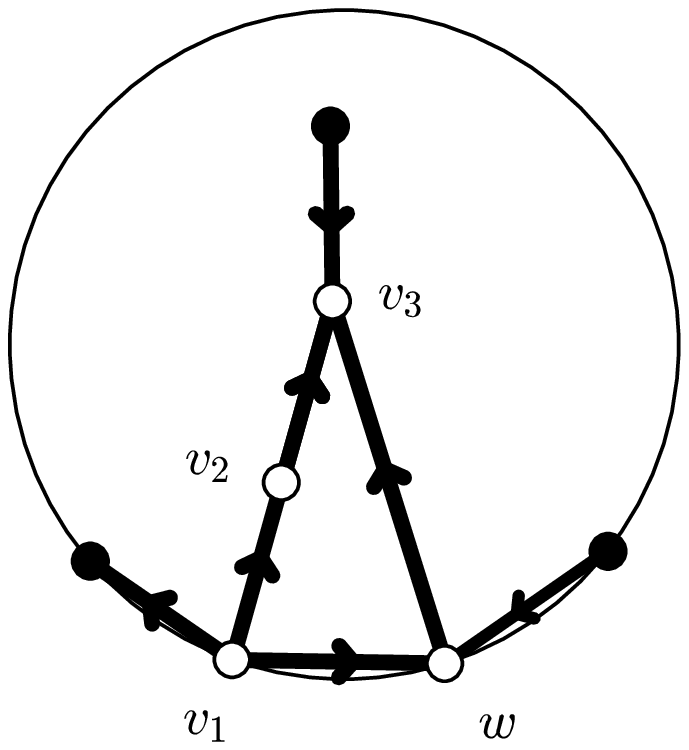}}
	\caption{ $X$  and $H$}
	\label{fig1}
\end{figure}

If  $v_3$ has two  sons, as depicted in Figure \ref{fig1},
then let $A$ be the matrix consisting two copies of columns of $A_H$ indexed by   $v_1v_2,   v_2v_3,   v_3w$ and one copy of the column of $A_H$ indexed by $v_1w$. I.e.,

 \begin{eqnarray*}
 A =
\begin{array}{ccccc}
v_1v_2, v_1v_2, v_2v_3, v_2v_3, v_3w, v_3w, v_1w\\
\begin{bmatrix}
-1&  -1&   0&   0&    0&   0&   -1   \cr
 0&   0&   1&   1&    0&   0&  -1    \cr
-1&  -1&   0&   0&    1&   1&    0   \cr
-1&  -1&   0&   0&    1&   1&   0    \cr
 0&   0&   1&   1&    1&   1&   0    \cr
 0&   0&   1&   1&    0&   0&  -1    \cr
 0&   0&   0&   0&    1&   1&   1    \cr
\end{bmatrix}&
\end{array}
\end{eqnarray*}
Then $\per(A) = -24$, and we are done.
\bigskip
\noindent

If  $v_3$ has three  sons,   as depicted in Figure \ref{fig4},
we choose  twice the columns of $A_H$ indexed by $v_1v_2,  v_2v_3,   v_3w, v_1w$.

\begin{figure}[h]
    \centering
    \scalebox{0.55}{\includegraphics{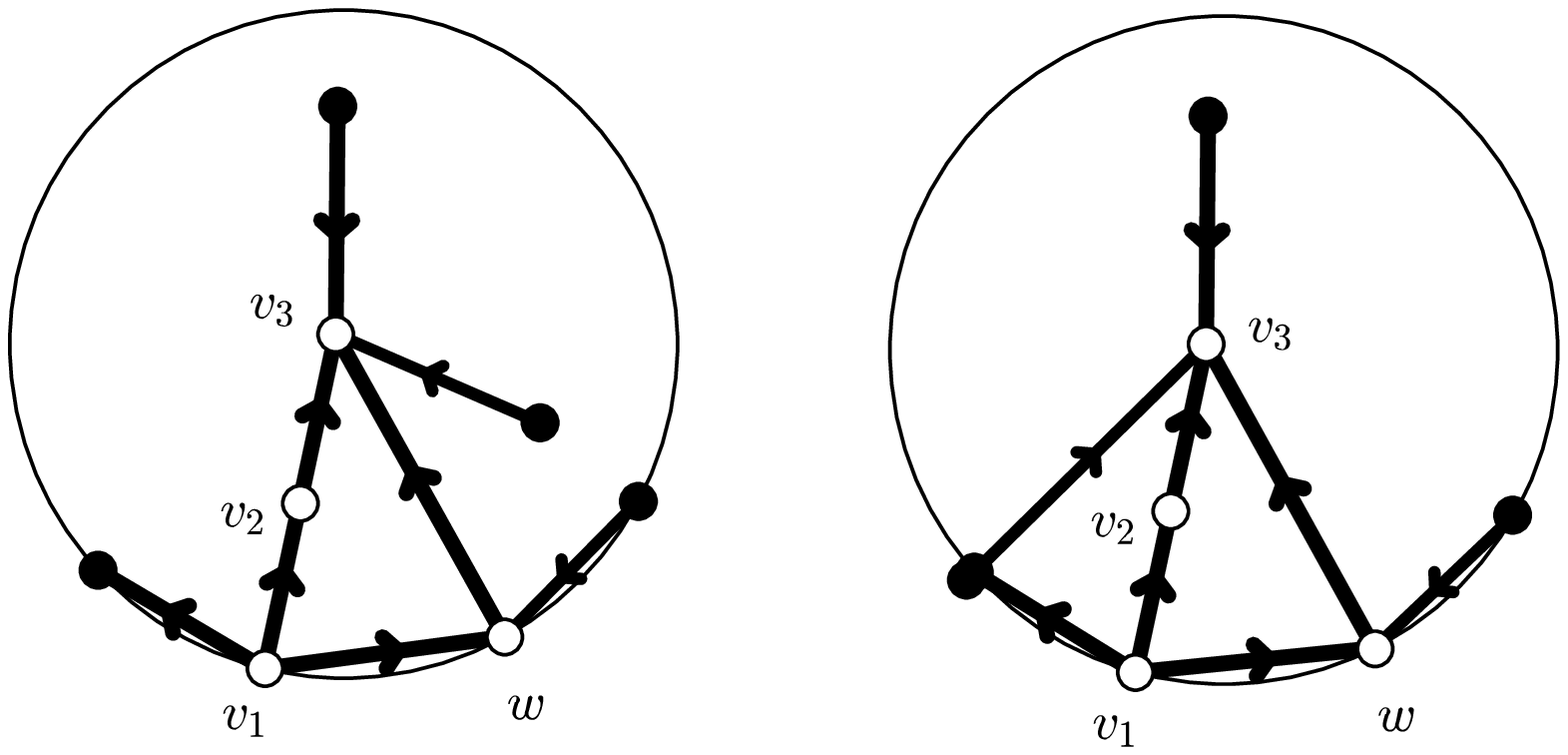}}
    \caption{  $X$  and $H$}
    \label{fig4}
\end{figure}

Then

 \begin{eqnarray*}
A =
\begin{array}{ccccc}
v_1v_2, v_1v_2, v_3w, v_3w, v_2v_3, v_2v_3, v_1w,v_1w\\
\begin{bmatrix}
-1& -1&  0&   0&  0&  0&   -1& -1 \cr
 0&  0&  0&   0&  1&  1&   -1& -1 \cr
-1& -1&  1&   1&  0&  0&    0&  0 \cr
-1& -1&  1&   1&  0&  0&    0&  0 \cr
 0&  0&  1&   1&  1&  1&    0&  0 \cr
 0&  0&  0&   0&  1&  1&   -1& -1 \cr
 0&  0&  1&   1&  0&  0&    1&  1 \cr
 0&  0&  1&   1&  1&  1&    0&  0
\end{bmatrix}&
\end{array}
\end{eqnarray*}
Then $\per(A)=-48$ and
we are done.

\bigskip
\noindent
{\bf Case 2:}  $v_3$ has at least four sons or $v_3$.
\bigskip

Let $X$ be the set consisting $v_3$ and all its descendants.  Let $Y=V-X$. Again, orient the edges of $H$ so that $v_3$ is a sink vertex, and all vertices at distance $2$ from $v_3$ are source vertices as in Figure \ref{fig2}.

\begin{figure}[h]
    \centering
    \scalebox{0.5}{\includegraphics{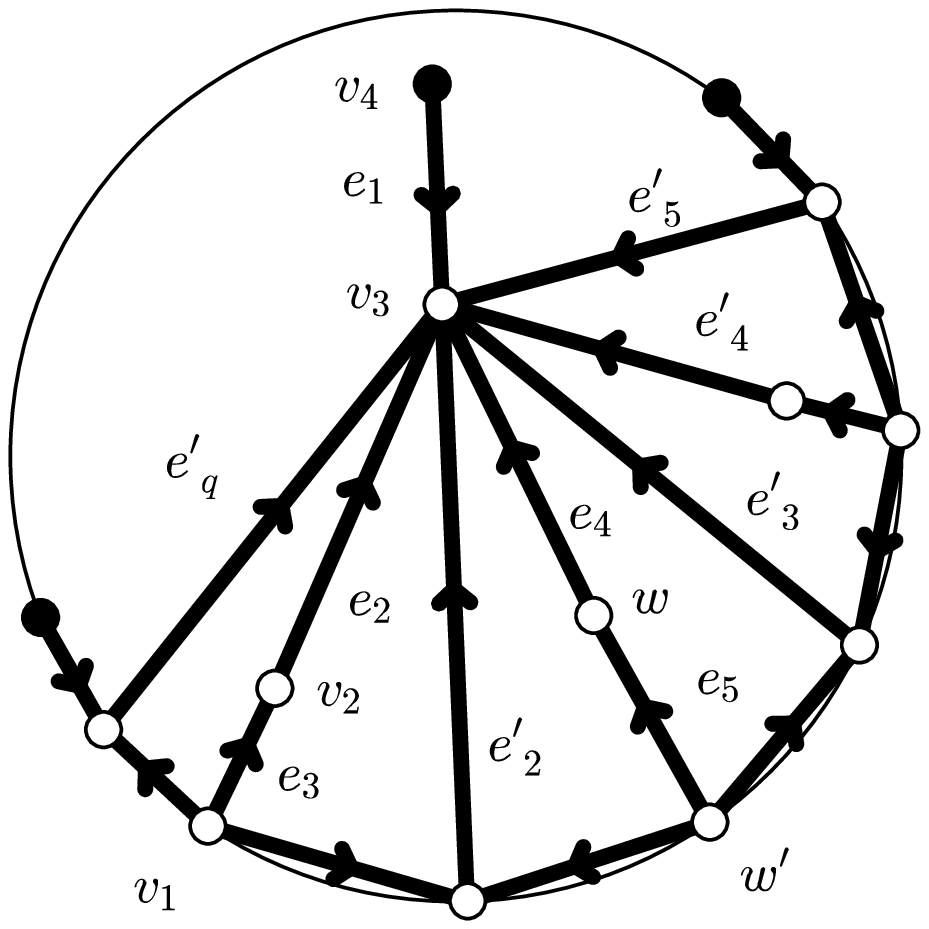}}
    \caption{ $X$  and $H$}
    \label{fig2}
\end{figure}

Then all the edges in $E[X] \cap E(T)$ are source or sink edges.

As $v_3$ has at least four sons, there is at least one son, say $w \ne v_2$, such that $w$ is not a leaf of $T$. Let $w'$ be the son of $w$ (note that since $G$ is bipartite, if a son of $v_3$ is not a leaf of $T$, then it has exactly one son), and $v_4$ be the father of $v_3$.
Let $e_1=v_3v_4$, $e_2=v_3v_2$, $e_3=v_2v_1$, $e_4=v_3w$ and $e_5=ww'$.

Each edge $e$ in $E[X] \cap E(T)$ is either incident to $v_3$ or is of the form $uu'$, where $u$ is a son of $v_3$ and $u'$ is the son of $u$.  Let  $\phi: E[X] \cup E[X,Y] \to E[X]$ be the mapping defined as follows:
 Cyclically order the edges of $E[X] - \{e_4\}$ incident to $v_3$ as $e'_1, e'_2, \ldots, e'_q$, and  $e'_1=e_2$. Let $\phi(e'_i) = e'_{i+1}$, where the indices are modulo $q$. Let $\phi(e_4)=e_5$. For each son $u$ of $v_3$ which is not a leaf of $T$, let $u'$ be the son of $u$, and let $\phi(uu')=uv_3$. In particular, $\phi(e_3)=e_2$ and $\phi(e_5)=e_4$.

Assume $e \in (E[X]  - E(T)) \cup E[X,Y]$. If $e=e_1$, then $\phi(e_1)=e_4$. Otherwise $e \in E(C)$, if $e$ is   to the left of $v_1$, then $\phi(e)$ is the tree edge incident to $e$ and to the right of $e$; if $e$ is to the right of $v_1$, then $\phi(e)$ is the tree edge incident to $e$ and to the left of $e$. (In particular, both cycle edges incident to $v_1$ are mapped to $e_3=v_1v_2$).

It is easy to verify that  for each $e \in E[X] \cup E[X,Y]$, $\phi(e)$ is   a source or a sink edge incident to $e \in E[X]$, and for each $e \in E[X]$, $|\phi^{-1}(e)| \le 2$. So it follows from Lemma \ref{lem-new} that  the subgraph $H=G[X] \cup E[X,Y]$ has a
$(0,2)$ matrix which contains no columns indexed by edges in $E[X,Y]$.

\bigskip
\noindent
{\bf Case 3:}  $v_3$ has only one son.
\bigskip

 Let $v_4$ be the father of $v_3$. If there is a son $w$ of $v_4$ with $d_T(w)\geq 3$, then  the depth of $w$ is  the same as $v_3$.
 We choose $w$ to play the role of $v_3$, and we are in   Cases 1 and 2.  Hence, we may assume that each son of $v_4$ has degree at most two in $T$.
 Let $X$ be the set consisting of $v_4$ and all the descendants of $v_4$ that have distance at most $2$ to $v_4$. Let $Y=V-X$. Orient the edges in $E[X] \cup E[X,Y]$ so that $v_4$ is a sink vertex, and all vertices at distance $2$ from $v_4$ are source vertices as in Figure \ref{fig3}.

\begin{figure}[h]
    \centering
    \scalebox{0.5}{\includegraphics{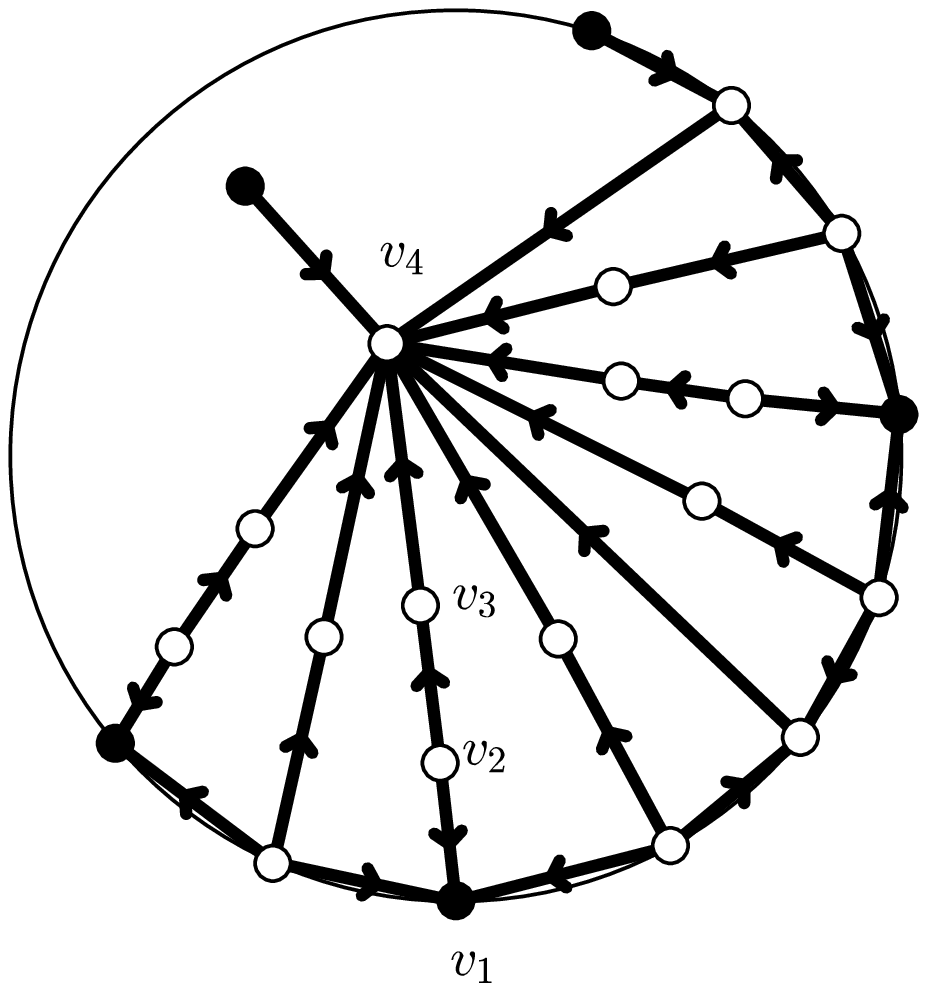}}
    \caption{ $X$  and $H$}
    \label{fig3}
\end{figure}

  In this orientation, all the edges in $E[X] \cap E(T)$ are source or sink edges. Similarly as in the previous case, it is easy to find a mapping $\phi: E[X] \cup E[X,Y] \to E[X]$
such that for each $e \in E[X] \cup E[X,Y]$, $\phi(e)$ is a source or a sink edge incident to $e$, and for each $e \in E[X]$,
$|\phi^{-1}(e)| \le 2$. By Lemma \ref{lem-new},    the subgraph $H=G[X] \cup E[X,Y]$ has a
$(0,2)$ matrix which contains no columns indexed by edges in $E[X,Y]$.

This completes the proof of Lemma \ref{bi}.
\end{proof}

It was proved in \cite{WZ2013}  that every graph $G$ has a \pnz
$(1,2)$-matrix. However, the following two conjectures which are weaker
than   Conjectures \ref{guess1} and  \ref{guess2}, respectively,  remain open.

\begin{conjecture}
\label{guess3} There is a constant $k$ such that every graph $G$  has a \pnz $(k,1)$-matrix.
\end{conjecture}

\begin{conjecture}
\label{guess4} There is a constant $k$ such that every graph $G$
without isolated edges has a \pnz $(0,k)$-matrix.
\end{conjecture}

{\bf{ Acknowledgement}}: This paper is finished while the 2nd author is visiting Professor Shinya Fujita at  Yokohama City University.  She thanks the hospitality of Professor Fujita and  Yokohama City University.

\end{document}